\documentclass[smallextended, numbook]{svjour3}

\usepackage{amssymb}
\usepackage{amsmath}
 \usepackage[dvips]{graphicx}
    \usepackage[usenames]{color}
\usepackage[top=1.5in,bottom=1.5in, left=1.1in, right=1.1in]{geometry}

\begin{document}

\title{Best proximity points for proximal contractions}
\author{Aurora Fern\'andez-Le\'on}

\institute{A. Fern\'andez-Le\'on \at Departamento de An\'alisis Matem\'atico, Universidad de Sevilla, Apdo. 1160, 41080 Sevilla, Spain\\\email{aurorafl@us.es\\Tel: 0034649160497}
}
\date{}
\maketitle

\begin{abstract}

In this paper we improve and extend some best proximity point results concerning the so-called proximal contractions. Specifically, compactness assumptions
under the sets $A$ and $B$ are removed to consider completeness conditions instead.


%
%

\end{abstract}

\newpage

%
%
%
%

\section{Introduction}
A classical and very well-studied problem in Metric
Fixed Point Theory is the existence of fixed point of single-valued
non-self mappings $T$ defined from a set $A$ to a set $B$, where $A$ and $B$ are subsets of a metric space $X$. Since the relation between these
two sets $A$ and $B$ may have so different natures ($A \cap B= \emptyset$, $A \cap B \neq \emptyset$, $A \subseteq B$, etc...), the branch of the theory that deals with these mappings has not
only studied the classical fixed point problem that tries to find solutions of
the equation $Tx=x$ but also has considered those situations where the existence of fixed points may even make no
sense. In these special cases, the interest of the theory mainly focuses on the existence of approximate solutions of the equation $Tx=x$, that is, points that are
close to its image somehow. In this direction, one
 classical well-known best approximation theorem due to Fan \cite{fan} is the one that states that if $A$ is a compact,
 convex and nonempty subset of a Hausdorff locally convex topological vector space $X$ and $T$ is a continuous mapping from $A$ to $X$, then there exists a point $x\in A$ such that $d(x,Tx)=d(Tx,A)$, where $d$ is the semi-metric induced by a continuous semi-norm defined on $X$ (notice that $d$ is a metric if the semi-norm is actually a norm). Both best approximation
theorems and best proximity point theorems are results in the theory that search approximate solutions, that is, points in $A$ (the domain of $T$) satisfying the condition $d(x,Tx)=d(Tx,A)$. The main difference between these two families of theorems is that the second one tries to go further in the approximation by searching approximate solutions which in turn are solutions of the minimization problem of the function $d(x,Tx)$ on $A$. In fact, best approximation results do not care whether the error $d(x,Tx)$ is the minimum possible in the context. To be precise, best proximity point theorems assure, for mappings $T$ defined from $A$ to $B$, the existence of approximate solutions that minimize the value $d(x,Tx)$ in the set $A$ and moreover with respect to $B$, i.e., guarantee the existence of points $x\in A$ such that $d(x,Tx) =d(A,B)$, where $d(A,B)$ is actually the absolute infimum of $d(x,Tx)$ in this context. Points that satisfy the previous equality are known in the theory either as absolute optimal approximate solutions or as best proximity points.

The study of the existence of best proximity points have been tackled in the theory from many different approaches. In this regard, generalizations of best proximity point theorems from self-maps to non-self-maps have arisen in the past two decades (see for example \cite{sha2,beer,kim,kire,save1,save2,seh1,srini}). Besides, a relevant research line is the one that seek best proximity points of cyclic mappings, i.e., self-mappings $T$ defined on the union of $A$ and $B$ and satisfying that $T(A)$ is contained in $B$ and $T(B)$ is contained in $A$. Many publications have arisen by assuming several additional metric conditions on these cyclic mappings, see for instance \cite{sha,diba,elve,elvek,espi,fer,kive,kole,suzu} and references therein. Our work developed here aims however to find best proximity points of classical non-self mappings that satisfy certain contractive conditions.

The present paper is basically motivated by the papers
\cite{Sadiq3} and \cite{Sadiq4}. These works include
existence, uniqueness and convergence results of best proximity
points for proximal contractions, thereby extending Banach's
contraction principle to the case of non-self mappings. Most of the results there considered
are proved by assuming some compactness properties under the sets $A$ and $B$
between which a proximal contraction $T$ is defined. The feeling that these hypothesis
are too restrictive in the context of proximal contractions has encouraged us to go into this
topic. In fact, the results we prove in the sequel come directly from an attempt to
improve in this direction the main results in \cite{Sadiq3} and \cite{Sadiq4}.



The work is organized as follows: in Section 2 we introduce most
of the definitions, notations and previous results we will need.
In Section 3, we find a weaker condition for a set $A$ than the fact of being approximatively compact
with respect to a set $B$ to guarantee the existence of best proximity point of proximal contractions of
the first and second kind. In this regard, we consider classical completeness assumptions as the fact
of being closed to improve the main results in \cite{Sadiq3} and \cite{Sadiq4}. Finally, we include an appendix where problems studied in Section \ref{sec2} are considered from a different approach.

\section{Preliminaries}\label{se2}

In this section we introduce the main concepts and results we will
need throughout this work. We begin by fixing some notations. Let
$(X,d)$ be a metric space and let $A$ and $B$ be two nonempty
subsets of $X$. Define
\begin{align*}
& d(x,A)=  \inf \{d(x,y) : y \in A\},\\
&d(A,B)= \inf \{d(x,y) : x\in A,y\in B\},\\
&A_0 =  \{x \in A : d(x, y) = d(A, B) \text{ for some } y \in B \},\\
&B_0 =  \{y \in B : d(x, y) = d(A, B) \text{ for some } x \in A \}.
\end{align*}
 From now on, $B(a,r)$ will denote the closed ball in the space $X$
centered at $a\in X$ with radius $r>0$.

Next, we define a notion of compactness between sets given in \cite{Sadiq3*}. This concept will play an essential role in this paper.

\begin{definition} Let $A$ and $B$ be two nonempty subsets of a metric space $X$.
$A$ is said to be approximatively compact with respect to $B$ if every sequence
$\{x_n\}$ of $A$ satisfying the condition that $d(y, x_n) \rightarrow d(y, A)$ for some $y \in B$ has a convergent
subsequence.
\end{definition}

It is worth mentioning that if a set $A$ is compact and moreover $B$ is approximatively compact with respect to $A$, then
the sets $A_0$ and $B_0$ are non-empty.

By the end of this paper we consider in several occasions some
geodesic spaces. Next, we include some background
concerning these metric spaces. A metric space
$(X,d)$ is said to be a {\it geodesic space } if every two points
$x$ and $y$ of $X$ are joined by a geodesic, i.e, a map
$c:[0,l]\subseteq {\mathbb R}\to X$ such that $c(0)=x$, $c(l)=y$,
and $d(c(t),c(t^{\prime}))=|t-t^{\prime}|$ for all $t,t^{\prime}
\in [0,l]$. Moreover, $X$ is called {\it uniquely geodesic } if
there is exactly one geodesic joining $x$ and $y$ for each $x,y\in
X$. When there is only one geodesic between two points $x$ and
$y$, the image of this geodesic (called geodesic segment) is
denoted by $[x,y]$. Thus, any Banach space is a
geodesic space with usual segments as geodesic segments.

A subset $A$ of a uniquely geodesic metric space $X$ is said to be
{\it convex} if the geodesic segment joining each pair of points
$x$ and $y$ of $A$ is contained in $A$. For more about geodesic
spaces the reader can check \cite{brha,bubu,papa}.

The notion of strict convexity in geodesic spaces was introduced in \cite{ahu} for strongly convex
metric spaces \cite{rol}, a wider family of metric spaces than the
one of geodesic metric spaces. Notice that this notion trivially
extends the strict convexity known in normed spaces to the nonlinear
setting.

\begin{definition} A geodesic metric space $X$ is said to be strictly
convex if for every $r>0$, $a,x$ and $y\in X$ with $d(x,a)\le r$,
$d(y,a)\le r$ and $x \neq y$, it is the case that $d(a,p)< r,$
 where $p$ is any point between $x$ and $y$ such that
$p\neq x$ and $p\neq y$, i.e., $p$ is any point in the interior of a
geodesic segment that joins $x$ and $y$.
\end{definition}
In these spaces, as it happens in strictly convex Banach spaces,
if $x$ and $y$ are any two points on the boundary of a ball then
the interior of the segment $[x,y]$ lies strictly inside the
ball. Moreover, it is easy to see that strictly convex metric spaces are in fact
uniquely geodesic.

\begin{example} Spaces of nonpositive curvature in the sense of Busemann
(see \cite{papa} for a detailed study on them) are strictly convex metric spaces.
Note that consequently the well-known CAT($0$) spaces are strictly convex metric spaces.
\end{example}

In the next section we will also mention the so-called nonempty
intersection property. Notice that this property extends the notion of reflexivity in
Banach spaces to geodesic metric spaces.

\begin{definition} Let $X$ be a uniquely geodesic metric space. $X$ is said to have
the nonempty intersection property if for any sequence $\{C_n\}$
of subsets of $X$ such that $C_{n+1} \subseteq C_n \textrm{  } \forall
n\in \mathbb{N}$ and $C_n$ is closed convex bounded and nonempty
$\textrm{  } \forall n\in \mathbb{N}$, it is the case that
$\displaystyle{\bigcap_{n\in {\mathbb N}} C_n }\neq \emptyset$.
\end{definition}

A broad family of geodesic spaces that satisfy this property is the
 one of uniformly convex metric spaces with either a monotone or lower
semicontinuous from the right modulus of convexity (see Definition 2.1 in \cite{esfe2} and Proposition
 2.2 in \cite{kole}). This uniform convexity in metric spaces is the result of assuming uniformity
 conditions on the definition of strict convexity given above and generalizes
 the notion of uniform convexity in Banach spaces. Thus, any uniformly convex Banach space is uniformly convex is this
sense.

Next we give the definitions of the mappings
we will deal with in the following section. These mappings were
introduced in the recent paper \cite{Sadiq4}. We begin with the definition of proximal contraction
of the first kind.

\begin{definition}\label{def1}
Let $A$ and $B$ be two nonempty subsets of a metric space $X$. A
map $T : A \rightarrow B$ is a proximal contraction of the first kind if
there exists a non-negative number $\alpha < 1$ such that
$$\left. \begin{array}{lc}
&d(u,Tx) = \text{d}(A,B)\\
&d(v,Ty) = \text{d}(A,B)
\end{array} \right \} \Rightarrow d(u,v)\leq \alpha d(x,y),$$
for all $u,v,x,y \in A$.
\end{definition}

It is immediate to see that these mappings are contractions and so continuous if we restrict to the case of self-mappings.
On the contrary, next example given in \cite{Sadiq3*} shows that these mappings are not continuous in general.

\begin{example} Let $A=[0,1]$ and $B=[2,3]$ subsets of $\mathbb{R}$ endowed with the Euclidean metric.
Then the mapping $T : A \rightarrow B$ defined as

$$f(x) = \left\{
\begin{array}{c l}
 3-x & if\text{ x is rational}\\
2+x & otherwise
\end{array}
\right.
$$
is a proximal contraction of the first kind.

\end{example}


Next we give the definition of proximal contraction of the second kind.

\begin{definition}
Let $A$ and $B$ be two nonempty subsets of a metric space $X$. A
mapping $T : A \rightarrow B$ is a proximal contraction of the second kind if
there exists a non-negative number $\alpha < 1$ such that
$$\left. \begin{array}{lc}
&d(u,Tx) = \text{d}(A,B)\\
&d(v,Ty) = \text{d}(A,B)
\end{array} \right \} \Rightarrow d(Tu,Tv)\leq \alpha d(Tx,Ty),$$
for all $u,v,x,y \in A$.
\end{definition}

Notice that unlike the other definition, these mappings may not be even continuous when
we restrict to the self case.

In the sequel, we also consider the following non-self mappings.

\begin{definition} Let $A$ and $B$ be two nonempty subsets of a metric space $X$. Let $g : A \rightarrow A$
be an isometry and $T : A \rightarrow B$ a non-self mapping. The mapping $T$ is said to preserve isometric distance with respect to $g$ if
$$d(Tgx,Tgy) =d(Tx,Ty)$$
for every $x$ and $y$ in $A$.
\end{definition}

A counterpart notion of fixed point in the context of non-self mappings is the so-called
best proximity point.

\begin{definition}
Let $T : A  \rightarrow  B$ be a non-self mapping where $A$ and $B$
are two nonempty subsets of a metric space $X$. A point $x \in A$ is said to be
a best proximity point for $T$ if $d(x,Tx)={\rm d}(A,B)$.
\end{definition}

\section{Main Results and Consequences}\label{sec2}

Our first result is a generalized best proximity point theorem for proximal contractions of the first kind. This theorem extends Theorem $3.1$ in \cite{Sadiq3} to more general settings by assuming weaker conditions
on the sets $A$ and $B$.

\begin{theorem}\label{ext1} Let $A$ and $B$ be two nonempty subsets of
a complete metric space $X$, $T : A \rightarrow B$ a proximal contraction of the first kind and $g : A \rightarrow A$ an isometry. Suppose that $T(A_0) \subseteq B_0$ and $A_0 \subseteq g(A_0)$. If $A_0$ is
nonempty and closed, then there exists a unique point $x \in A$ such that $d(gx,Tx)=d(A,B)$. Moreover, for every $x_0 \in A_0$ there exists a sequence $\{x_n\} \subseteq A$ such that
$d(g x_{n+1},Tx_n)=d(A,B)$ for every $n \geq 0$ and satisfying $x_n \rightarrow x$.
\end{theorem}

\begin{proof} Let $x_0$ be a point in $A_0$. Since $T(A_0) \subseteq B_0$ and $A_0 \subseteq g(A_0)$,
we can find a point $x_1 \in A_0$ such that $d(gx_1,Tx_0)=d(A,B)$. For this point
$x_1 \in A_0$, we can proceed similarly so that we get another point $x_2 \in A_0$ satisfying the condition $d(gx_2,Tx_1)=d(A,B)$. By repeating this process, once we have $x_n \in A_0$, it is possible to find $x_{n+1} \in A_0$ such that $d(gx_{n+1},Tx_n)=d(A,B)$ for every natural number $n$.

Since $T$ is a proximal contraction of the first kind,

$$d(gx_{n+1},gx_n) \leq \alpha d(x_n, x_{n-1})$$
for every $n\in \mathbb{N}$. The isometric character of $g$ shows that in fact

$$d(x_{n+1},x_n) \leq \alpha d(x_n, x_{n-1})$$
for every $n\in \mathbb{N}$. Thus, from this inequality we deduce that $\{x_n\}$ is a Cauchy sequence.
By using the completeness of the space $X$ and the fact that $A_0$ is closed, we conclude that $x_n \rightarrow x \in A_0$.

In addition, the continuity of $g$ implies that $gx_n \rightarrow gx$. Since $gx_n \in A_0$ for every $n\in \mathbb{N}$,
$gx$ is also in the set $A_0$.

Since $x \in A_0$, we have that $Tx \in B_0$ and therefore there exists a point $z \in A_0$ such
that $d(Tx,z)=d(A,B)$. By applying again that $T$ is a proximal contraction, we obtain that
$$d(z,gx_{n+1}) \leq \alpha d(x,x_n).$$
Now, by letting $n$ go to infinity, we have that $d(z,gx_{n+1}) \rightarrow 0$ and consequently $z=gx$. Thus, we conclude
that $d(gx,Tx)=d(A,B)$.

Now the proof follows the same patterns as the proof of \cite[Theorem 3.1]{Sadiq3}. Let $\hat{x}$ be another point in $A_0$ such that
$$d(g\hat{x},T\hat{x})=d(A,B).$$
Then, again by using the properties of $T$ and $g$ and reasoning as above, we get that
$$d(gx,g\hat{x})=d(x,\hat{x}) \leq \alpha d(x,\hat{x}).$$
Consequently, $\hat{x}$ must be equal to $x$ and thus $x$ is the only point in $A_0$ such that $d(gx,Tx)=d(A,B)$, which completes
the proof.
\end{proof}

Next, by considering the mapping $g$ as the identity map on $A$, we get the following consequence. Notice also that this
corollary provides existence and uniqueness of best proximity points in a general metric space for proximal contractions
of the first kind.

\begin{corollary}\label{co1} Let $A$ and $B$ be two nonempty subsets of
a complete metric space $X$ and $T : A \rightarrow B$ a proximal
contraction of the first kind.
Suppose that $T(A_0) \subseteq B_0$. If $A_0$ is
nonempty and closed, then there exists a unique best
proximity point $x \in A$. Moreover, for every $x_0 \in A_0$ there exists a sequence $\{x_n\} \subseteq A$ such that
$d(x_{n+1},Tx_n)=d(A,B)$ for every $n \geq 0$ and satisfying $x_n \rightarrow x$.
\end{corollary}

These two previous results notably improve those
given in \cite{Sadiq3} as Theorem $3.1$ and Corollary $3.1$.
In \cite{Sadiq3}, the set $B$ is assumed to be approximatively compact
with respect to $A$. This topological property, besides being very restrictive, implies
that the set $A_0$ is closed, which is the property required in the present paper. Next proposition easily shows
this assertion.

\begin{proposition}\label{relacion} Let $A$ and $B$ two nonempty subsets of a metric space $X$. If $A$ and $B$ are closed and $A$ is approximatively compact
with respect to $B$, then $B_0$ is closed.

\end{proposition}
\begin{proof} Let $\{y_n\} \subseteq B_0$ a convergent sequence and $y \in X$ its limit. The fact that $B$ is closed implies that $y\in B$. On
the other hand, since the point $y_n$ is in $B_0$,
there exists $x_n \in A_0$ such that $d(y_n,x_n)=d(A,B)$ for every $n\in \mathbb{N}$.
As a consequence, we get that
$$d(x_n,y) \rightarrow d(A,B)$$
when $n$ goes to infinity. Since $A$ is closed and approximatively compact
with respect to $B$, we also have that there exists a subsequence
of $\{x_n\}$ that converges to a point $z\in A$. Finally,
by combining the information we have, we get that $d(z,y)=d(A,B)$, which implies
that $y \in B_0$.
\end{proof}

Next example shows that the opposite implication of the previous proposition is not true. As a consequence of
this fact, we get that Theorem \ref{ext1} and Corollary \ref{co1} are actually strictly more general than the counterpart
results in \cite{Sadiq3}.

\begin{example}\label{extomas} Let $X$ denote the Banach space given by the set $\ell_2$ endowed with the norm
$$\|x\|=max\{\|x\|_2,\sqrt2\|x\|_\infty\}.$$
Notice that this norm is equivalent to $\|\cdot\|_2$ and therefore $(X,\|\cdot\|)$ is also
a reflexive Banach space.

Let $A=\{x = (x_n) \in X : x_1 = 1 \text{ and } \|x\| \leq \sqrt2\}$ and $B=\{z=(2,0,0, \ldots)\}$.
It is easy to see that both $A$ and $B$ are closed and convex. Notice that $e_1 \in A$, $\|z-e_1\|=\sqrt2$ and
moreover, if $y\in A$ we have that $y=(1,y_2,y_3, \ldots)$ with $\sum_{n \geq 2} y_n^2 \leq 1$ and
$|y_n| \leq 1$ for every $n\in \mathbb{N}$. As a consequence, we have that $\|z-y\|_2 \leq \sqrt2$ and $\|z-y\|_\infty =1$ for every $y\in A$ and
therefore $d(A,B)=\sqrt2=d(z,y)$ for every $y\in A$.

Let $\{u_n\}$ be the sequence in $A$ given by $u_n=e_1+e_n$. Since this sequence
does not have any convergent subsequence, we conclude that $A$ is not approximatively compact
with respect to $B$.
\end{example}

It is worth mentioning that in general the set $A_0$ is closed whenever the metric space $X$ is
a reflexive Banach space and the sets $A$ and $B$ are closed with $A$ convex. Therefore, Example \ref{extomas} is just a particular case in this respect where $A$ is not approximatively compact with respect to $B$ and $B_0$ is closed.

The next result deals with proximal contractions of the second kind. Again, a compactness assumption is removed, now from
Theorem $3.1$ in \cite{Sadiq4}, to consider a completeness hypothesis. Also notice that here we do not impose any continuity property on $T$.

\begin{theorem} \label{ext2} Let $A$ and $B$ be two nonempty subsets of
a complete metric space $X$, $T : A \rightarrow B$ a proximal contraction of the second kind that preserves isometric distance with respect to an isometry $g : A \rightarrow A$. Suppose that $T(A_0) \subseteq B_0$ and $A_0 \subseteq g(A_0)$. If $T(A_0)$ is
nonempty and closed, then there exists a point $x \in A$ such that $d(gx,Tx)=d(A,B)$. Moreover, if $\hat{x}$ is another point in $A$
for which $d(g\hat{x},T\hat{x})=d(A,B)$, then $Tx=T\hat{x}$.
\end{theorem}
\begin{proof} Following similar patterns to those considered in Theorem \ref{ext1}, we have
that there exists a sequence $\{x_n\} \subseteq A$ satisfying
\begin{equation}\label{1}
d(gx_{n+1},Tx_n)=d(A,B).
\end{equation}
Since $T$ is a proximal
contraction of the second kind,
$$
d(Tgx_{n+1},Tgx_n) \leq \alpha d(Tx_{n-1},Tx_n)
$$
for every $n\in \mathbb{N}$. By using the isometric character of $g$ and the fact that $T$ preserves
isometric distance with respect to $g$, we get
$$d(Tx_{n+1},Tx_n) \leq \alpha d(Tx_n, Tx_{n-1})$$
for every $n\in \mathbb{N}$. Thus, as a consequence, we conclude that $\{Tx_n\}$ is a Cauchy sequence. By using the completeness of the space $X$ and the fact that $T(A_0)$ is closed, we obtain that $\{Tx_n\}$ converges to some point $y$ in $T(A_0) \subseteq B_0$. Hence, $y=Tu$ for some $u \in A_0$. Moreover, we also deduce that there exists $z\in A_0$ such that $d(z,Tu)=d(A,B)$. Since $A_0 \subseteq g(A_0)$, we have that
$z=gx$ for some $x \in A_0$. Then, $d(gx,Tu)=d(A,B)$. By using (\ref{1}) and previous equality, we get
$$d(Tgx,Tgx_{n+1})=d(Tx,Tx_{n+1}) \leq \alpha d(Tu,Tx_n).$$
Consequently, we obtain that $Tx=Tu$ and hence $d(gx,Tx)=d(A,B)$. The uniqueness of such a point $x$ follows similar patterns to those considered in \cite[Theorem 3.1]{Sadiq4}. Let $\hat{x}$ be another point in $A_0$ such that
$$d(g\hat{x},T\hat{x})=d(A,B).$$
Then, again by using the properties of $T$ and $g$ and reasoning as above, we get that
$$d(Tgx,Tg\hat{x})=d(Tx,T\hat{x}) \leq \alpha d(Tx,T\hat{x}).$$
Consequently, $T\hat{x}$ must be equal to $Tx$, which completes
the proof.
\end{proof}

\begin{remark} Notice that, in general, if a second proximal contraction $T$ is also continuous and $B$ is approximatively compact with respect to $A$, then $T(A_0)$ is sequentially closed in $B_0$ for sequences $\{Tx_n\} \subseteq T(A_0)$ such that
$d(gx_{n+1},Tx_n)=d(A,B)$ for every $n\in \mathbb{N}$, that is what has essentially been considered in the proof of previous theorem.

\end{remark}

If we consider again the particular case of $g$ equal to the identity map, we obtain an
improvement of Corollary $3.2$ of \cite{Sadiq4} as well as a best proximity point result for proximal contractions of the second kind.

\begin{corollary} Let $A$ and $B$ be two nonempty subsets of
a complete metric space $X$ and $T : A \rightarrow B$ a proximal
contraction of the second kind.
Suppose that $T(A_0) \subseteq B_0$. If $T(A_0)$ is
nonempty and closed, then there exists a best proximity point $x \in A$. Moreover, if $\hat{x}$ is another best proximity point in $A$,
 then $Tx=T\hat{x}$.

\end{corollary}

\section{Appendix: A Property That Involves the Pair $(A,B)$.}

One interesting problem that arises when dealing with Section \ref{sec2} is whether it is possible to establish or formulate
its results by using a completeness property that involves both set $A$ and $B$ in such a way that we preserve the classical structure of recent theorems concerning best proximity points (see, apart from the property of being approximatively compact with respect to another set, the properties UC, WUC or HW considered when dealing with best proximity points in cyclic mappings \cite{esfe2,fer,suzu}). In this regard, we
introduce a useful property involving the sets $A$ and $B$.

\begin{definition} Let $A$ and $B$ be nonempty subsets of a metric space
$(X,d)$. The pair $(A,B)$ is said to satisfy property WAC if for
every sequence $\{ x_n\}$ in $A$ and every
point $p$ in $B$ such that
$$\lim_nd(x_n,p)= d(A,B),$$
then it is the case that $p \in B_0$.
\end{definition}

Next proposition shows that this property is weaker than the approximative compactness considered in Section \ref{sec2}.

\begin{proposition} Let $A$ and $B$ be nonempty subsets of a metric space
$(X,d)$. If $A$ is approximatively compact with respect to $B$ and $A$ is closed, then the pair $(A,B)$ has property WAC.
\end{proposition}

\begin{proof} We omit the proof.

\end{proof}

The following proposition establishes some general situations where we find property WAC between sets.

\begin{proposition}Let $A$ and $B$ be two nonempty subsets of a metric space $X$.

\begin{itemize}

\item [(1)] If $X$ is a reflexive Banach space and the sets $A$ and $B$ are such that $A$ is closed and convex and $B$ is closed, then $(A,B)$ has property
WAC.

\item [(2)] If $X$ is a strictly convex metric space with the
nonempty intersection property (see Section \ref{se2} for definitions) and the sets $A$ and $B$ are such that $A$ is closed and convex and $B$ is closed, then $(A,B)$ has property
WAC.

\end{itemize}
\end{proposition}
\begin{proof} $(1)$ Let $\{x_n\} \subseteq A$ and $p \in B$ such that $\lim_nd(x_n,p)= d(A,B)$. Since $\{x_n\}$ is bounded, we have that
 $x_n$ weakly converges to a point $x \in A$. By using that fact that the norm in $X$ is weakly lower semicontinuous, we also get that
  $d(x,p)\leq \liminf_{n} d(x_n,p)=d(A,B)$. Consequently, we conclude that $p \in B_0$.

$(2)$ Let $\{x_n\} \subseteq A$ and $p \in B$ such that $\lim_nd(x_n,p)= d(A,B)$. Let us consider for
every $n\in \mathbb{N}$ the subsets $C_n$ of $A$ given by $C_n =A \cap B(p,d(A,B)+1/n)$. Since balls are convex in the context of strictly convex metric spaces, we may assert that $C_n$ are closed, convex and bounded. Moreover, it is not difficult to see that for every $n\in \mathbb{N}$ there exists $m_0 \in \mathbb{N}$ such that $x_m \in C_n$ for every $m \geq m_0$. Consequently we conclude that $C_n$ is also nonempty for every natural number $n$. By using the fact that $X$ has the nonempty intersection property we get that $\bigcap C_n \neq \emptyset$.
 Let $x$ be a point in this intersection. Since $x\in B(p,d(A,B)+1/n)$ for every $n\in \mathbb{N}$, we get that $d(p,x)=d(A,B)$ and thus $p \in B_0$.
\end{proof}

As a consequence of the first statement of this proposition, we deduce that the pair of sets $(A,B)$ considered in Example \ref{extomas} also have property WAC and therefore property WAC is strictly weaker than approximative compactness.

\begin{remark} For those interested in properties UC, WUC or HW, notice that in general property WUC implies approximative compactness (and thus property UC too). Then, property WAC is more general than properties WUC and UC. However, it is possible to find pairs of sets with property HW and without property WAC. For this aim, it suffices to consider $A=\{-1/n : n \in \mathbb{N}\}$ and $B=\{1+1/n : n\in \mathbb{N}\} \cup \{$$1 \}$.
\end{remark}

The proposition and the example given below show that property WAC is strictly stronger than the fact that $A_0$ is closed. For this reason, we have not consider property WAC in Section \ref{sec2}.

\begin{proposition} Let $A$ and $B$ be two nonempty subsets of a metric space. If $A$ is closed and the pair $(B,A)$ has property WAC, then $A_0$ is closed.

\end{proposition}
\begin{proof} The proof follows similar patterns to those considered in Proposition \ref{relacion}.
\end{proof}

Next example shows that the inverse implication of the previous proposition is not true.

\begin{example} Let $X$ be the Euclidean plane. Let $A=\{(x,1) : x \geq 1\} \cup \{(x,0) : x \geq 1\} $ and $B=\{(0,y) : y\neq 0 \}$. Consider the sequence $\{y_n\} \subseteq B$ given by $\{y_n\}=\{(0,-1/n)\}_{n\in \mathbb{N}}$ and the point $q=(1,0) \in A$. Notice that
$A_0=\{(1,1)\}$, so it is closed, but however $(B,A)$ does not have property WAC since $d(y_n,q)\rightarrow 1=d(A,B)$ but $q$ is not in $A_0$.

\end{example}

\begin{acknowledgements}
The author was partially supported by the Ministery of Science and Technology of Spain, Grant
MTM2009-10696-C02-01 and La Junta de Antaluc´ia, Grant P08-FQM-03543.

\end{acknowledgements}

%

%
%
%

\end{document}